\documentclass[a4paper, 12pt]{amsart}

\usepackage{amssymb}
\usepackage{amsmath}
\usepackage{amscd}
\usepackage{amsthm}
\usepackage[centertags]{amsmath}
\usepackage{amsfonts}
\usepackage{newlfont}
\usepackage[all]{xy}
\usepackage{graphicx}
\usepackage{amsfonts, amssymb}
\usepackage[usenames]{color}
\usepackage{mathrsfs}
\usepackage{latexsym}

\newtheorem{thm}{Theorem}[section]

\newtheorem{lem}[thm]{Lemma}
\newtheorem{prop}[thm]{Proposition}

\theoremstyle{definition}
\newtheorem{definition}[thm]{Definition}
\newtheorem{rem}[thm]{Remark}

\newcommand{\ccc}{\mathbf c}

\newcommand{\zC}{\mathbb C}

\newcommand{\zR}{\mathbb R}
\newcommand{\zN}{\mathbb N}
\newcommand{\zK}{\mathbb K}
\newcommand{\zE}{\mathbb E}

\newcommand{\ve}{\varepsilon}

\begin{document}
\baselineskip=.75cm

\title[On the linear polarization constants of finite dimensional spaces]{On the linear polarization constants of finite dimensional spaces}

\author[D. Carando]{Daniel Carando}
\address{Departamento de Matem\'{a}tica - Pab I,
Facultad de Cs. Exactas y Natu\-ra\-les, Universidad de Buenos Aires,
(1428) Buenos Aires, Argentina and IMAS-CONICET}
\email{dcarando@dm.uba.ar}

\author[D. Pinasco]{Dami\'an Pinasco}
\address{Departamento de Matem\'{a}ticas y Estad\'{i}stica, Universidad Torcuato Di Tella, Av. Figueroa Alcorta 7350, (1428) Buenos Aires, Argentina and CONICET}
\email{dpinasco@utdt.edu}

\author[J. T. Rodr\'{i}guez]{Jorge Tom\'as Rodr\'{i}guez}
\address{Departamento de Matem\'{a}tica and NUCOMPA, Facultad de Cs. Exactas, Universidad Nacional del Centro de la Provincia de Buenos Aires, (7000) Tandil, Argentina and CONICET}
\email{jtrodrig@dm.uba.ar}

\begin{abstract}
We study the linear polarization constants of finite dimensional Banach spaces. We obtain the correct asymptotic behaviour of these constants for the spaces $\ell_p^d$: they behave as $\sqrt[p]{d}$ if $1\le p\le 2$ and as $\sqrt{d}$ if $2\le p<\infty$. For $p=\infty$  we get the asymptotic behaviour up to a logarithmic factor.
\end{abstract}

\maketitle

\section*{Introduction}

Given a Banach space $X$, its \textit{$n$th linear polarization constant} is defined as the smallest constant $\ccc_n(X)$ such that for any set of $n$ linear functionals $\{\psi_j\}_{j=1}^n\subseteq X^*$, we have
\begin{equation}\label{problempolarizacion}
\Vert \psi_1 \Vert \cdots \Vert \psi_n \Vert  \le \ccc_n(X) \, \Vert \psi_1 \cdots \psi_n\Vert,
\end{equation}
where $\psi_1 \cdots \psi_n$ is the $n$-homogeneous polynomial given by the pointwise product of $\psi_1,\ldots,\psi_n,$ and $\Vert \cdot \Vert$ is the supremum norm over the unit sphere of $X.$
Related to this concept the \textit{linear polarization constant }$\ccc(X)$ of $X$ is defined as
$$\ccc(X) = \displaystyle\lim_{{n\rightarrow \infty}} (\ccc_n(X))^{\frac 1 n}.$$
The existence of this limit is a result of \cite{RS}.

These constants have been studied by several authors. Among the works on this topic, in~\cite{RT} the authors proved that for each $n$ there is a constant $K_n$ such that $\ccc_n(X) \leq K_n$ for every Banach space $X$. As a corollary of Theorem 3 from \cite{BST} the best possible constant $K_n$, for \textit{complex} Banach spaces, is $n^n$.
Arias-de-Reyna proved in~\cite{A}  that if $X$ is a \textit{complex }Hilbert space, of dimension greater or equal than $n$,  then
$$\ccc_n(X)=n^{\frac n 2}.$$
This result holds for real Hilbert spaces and $n\leq 5$ (see Theorem 4.6 in \cite{PPT}), but it is not known if it is true for every natural number $n$.

We recall that the linear polarization constant is infinite for infinite dimensional Banach spaces (see Theorem 12 in \cite{RS}). As a consequence, an interesting problem is to understand how this constant behaves as the dimension of the involved spaces vary. For example, the linear polarization constant of a real $d$-dimensional Hilbert space $\mathcal{H}_d$ was obtained by \mbox{Garc\'ia-V\'azquez} and Villa  in \cite{GV}, where they proved that $\ccc(\mathcal{H}_d)$ behaves like $\sqrt{d}$ as $d$ goes to infinity. This result was later extended to complex Hilbert spaces by  A. Pappas and S. G. R\'ev\'esz in \cite{PR}. For the spaces $\ell_1(\zC^d)$ it is know that $\ccc(\ell_1^d(\zC))=d$ (see Proposition 17 of \cite{RS}).

In this article we study the $n$th linear polarization constants, as well as the linear po\-la\-ri\-za\-tion constant of finite dimensional Banach spaces. In the first section we develop a method to estimate the linear polarization constant of a finite dimensional space (see Theorem~\ref{mainprop}). In Section~\ref{sec-elepe}, we apply this method to the finite dimensional spaces $\ell_p^d(\zK)$, obtaining in Theorem~\ref{teo polarizacion} the following asymptotically  optimal results on $d$ (the \emph{asymptotic} notation is explained in Section~\ref{sec-elepe}):
$$\ccc(\ell_p^d(\zK))											
          \asymp  \sqrt[p]{d} \quad \mbox{ if }  1\le p \leq 2    \quad \mbox{ and }	\quad \ccc(\ell_p^d(\zK))  \asymp \sqrt{d}  \quad \mbox{ if }    2\le p<\infty . $$
For $p=\infty$ we obtain $\sqrt{d} \prec  \ccc(\ell_\infty^d(\zK)) \prec \sqrt{d \log d}.$

In Section~\ref{seccion infinito} we use a probabilistic approach to estimate the norm of the product of linear functionals with coefficients $\pm 1$ (in the canonical basis) over  the spaces $\ell_\infty^d(\zC)$. This allow us to  give in Proposition~\ref{prop infinito} some estimates for their $n$th linear polarization constants.



\section{Linear polarization constants of finite dimensional spaces}

Throughout this work, given a Banach space $X$, $B_X$ and $S_X$ will stand for the unit ball and the unit sphere respectively.

In this section we present a general method for es\-ti\-ma\-ting linear polarization constants. In order to state our results, which give lower and upper bounds for these constants, we will define the so--called \emph{admissible} measures, which are measures that satisfy a rather mild condition.

\begin{definition}\label{admisible} Let $X$ be a Banach space and $\lambda$ a Borel measure over a Borel subset $K\subseteq B_X$. We say that $\lambda$ is \textit{admissible} if
$$\int_K \log|\langle x,\psi\rangle|\,\,\,d\lambda(x)$$
is finite for every $\psi\in S_{X^*}$ and the functions $g_m:S_{X^*}\rightarrow \zR$ defined as
\[
g_m(\psi)=\int_K \max\{\log|\langle x,\psi\rangle |,-m\} \,\,\,d\lambda(x),
\]
converges uniformly to the function $g:S_{X^*}\rightarrow \zR$, defined as
\[
g(\psi)=\int_K \log|\langle x,\psi\rangle |\,\,\,d\lambda(x).
\]
\end{definition}

For example, for $\mathcal{H}$ a finite dimensional Hilbert space, the Lebesgue measure over $S_\mathcal{H}$ is admissible, since the functions $g_m$ are constant functions that converges to the constant function $g$.

The main result of this section is the following.

\begin{thm}\label{mainprop} Given a finite dimensional Banach space $X$, let $\mu$ and $\eta$ be admissible probability measures over $S_X$ and $S_{X^*}$ respectively. Then there is $\psi_0 \in S_{X^*}$ and $x_0\in S_X$, depending on $\mu$ and $\eta$, such that
$$ \exp\left\{-\int_{S_{X^*}} \log |\langle x_0,\psi\rangle | \,\,\,d\eta(\psi) \right\} \leq \ccc(X)
 \leq  \exp\left\{- \int_{S_X} \log |\langle x,\psi_0\rangle | \,\,\,d\mu(x) \right\}.$$
\end{thm}

We will treat separately the lower and the upper bound, and state both as propositions. Let us first sketch some of the ideas behind the proof, specifically for the lower bound. Since $X$ is finite dimensional, by a compactness argument there exist, for each natural number $n$, linear functionals $\psi_1^n,\ldots,\psi_n^n\in S_{X^*}$ such that
\begin{equation}\label{peores}
\Vert  \psi_1^n\cdots \psi_n^n \Vert =\ccc_n(X)^{-1}.
\end{equation}
Take now $x_n\in B_X$ a point where the function $\psi_1^n\cdots \psi_n^n$ attains its norm, i.e., $\Vert  \psi_1^n\cdots \psi_n^n \Vert = \vert  \psi_1^n\cdots \psi_n^n (x)\vert$. Then,
$$ \Vert  \psi_1^n\cdots \psi_n^n \Vert^{\frac{1}{n}}= \exp \left\{\frac{1}{n} \sum_{i=1}^n \log |\psi_i^n(x_n)|\right\}.$$

If we consider the functions $f_n:S_{X^*}\rightarrow \zK$ defined as $f_n(\varphi)=\log|\varphi(x_n)|$ and $\eta_n$  the probability measure over $S_{X^*}$ defined as
$$\eta_n = \frac{1}{n} \sum_{i=1}^n \delta_{\psi_i^n},$$
then we have:
$$\frac{1}{n}\sum_{i=1}^n \log\left |\psi_i^n(x_n)\right|= \int_{S_{X^*}}	f_n(\psi) \,\,\,d\eta_n .$$

The idea now is to take a subsequence $\{n_k\}$ such that $\eta_{n_k}$ $w^*$-converges to some probability measure $\eta$ and such that $x_{n_k}$ converges to some $x_0\in S_X$. All this will give us an estimate of $\ccc(X)$ in terms of $\eta$ and the function $f_0:S_{X^*}\rightarrow \zK$, defined as $f_0(\varphi)=\log|\varphi(x_0)|$.

Since it is not clear how to find
a set of functions satisfying \eqref{peores} (and then, it is not clear that we can obtain $\eta$ and the estimate for $\ccc$), the following alternative procedure gives a lower bound for it: we fix a measure $\eta$ beforehand and choose the sets of linear functionals $\psi_1^n,\ldots , \psi_n^n$ to obtain this particular $\eta$ as the $w^*$-limit of the measures $\eta_n$. These sets of linear functionals may not satisfy  \eqref{peores}, but we clearly have
$$ \Vert  \psi_1^n\cdots \psi_n^n \Vert \geq \ccc_n(X)^{-1},$$
which is precisely what we need to obtain the desired lower bounds.
The sharpness of the bounds thus obtained will depend on the  good choice of the probability measure $\eta$.


\subsection*{Lower bounds in Theorem \ref{mainprop}}

In the sequel, for a measure space $(K, \nu)$ and an integrable function $f:K\rightarrow \zR$ we will use the notation
$$\nu(f) = \int_K f(\omega) \,\,\,d\nu(\omega).$$

We need the following auxiliary lemma due to  A. Pappas and S. G. R\'ev\'esz (see \cite[Lemma~4]{PR}).

\begin{lem}\label{stronglaw}  Let $\eta$ be any probability measure over $S_{X^*}.$ There is a sequence of sets of norm one linear functionals $\{\psi_1^n,\ldots, \psi_n^n\}_{n\in \zN}$ over $X$  such that

$$\displaystyle\lim_{n\rightarrow \infty} \frac 1 n \sum_{j=1}^n f(\psi_j^n) = \int_{S_{X^*}} f(\psi) \,\,\,d\eta(\psi)$$
for any continuous function $f:S_{X^*} \rightarrow \zR$. In other words, if we consider the measures $\eta_n=\frac{1}{n}\sum_{j=1}^n \delta_{\psi_j^n}$, the sequence $\{\eta_n\}_{n\in \zN}$ $w^*$-converges to $\eta$.
\end{lem}

We remark that, although the result in \cite{PR} is stated for $X$ a Hilbert space and $\eta$ the normalized Lebesgue measure, the proof works in the more general setting of our statement.
Now we are ready to prove the lower estimates for $\ccc(X)$.

\begin{prop} Given a finite dimensional Banach space $X$ and an admissible probability measure $\eta$ over $S_{X^*}$, there is a point $x_0 \in S_{X}$, depending on $\eta$, such that
$$\ccc(X) \geq \exp\left\{-\int_{S_{X^*}} \log |\langle x_0,\psi\rangle | \,\,\, d\eta(\psi) \right\}.$$
\end{prop}

\begin{proof} Take a sequence of sets of norm one of linear functionals $\{\psi_1^n,\ldots,\psi_n^n\}_{n\in \zN}$ as in Lemma~\ref{stronglaw}, and consider the measures $\eta_n=\frac{1}{n}\sum_{j=1}^n \delta_{\psi_j^n}$. Let $x_n\in S_X$ be a point where  $\prod_{j=1}^n \psi_j^n$ attains its norm. We may assume $\Big\Vert \prod_{j=1}^n \psi_j^n \Big\Vert^\frac{1}{n}$ converges, otherwise we work with a subsequence. With the same argument we may assume that there is $x_0\in S_X$ such that $x_n\rightarrow x_0$.
Since
$$c_n(X) \Big\Vert \prod_{j=1}^n \psi_j^n \Big\Vert \geq 1,$$ we need an upper bound for $ \displaystyle\lim_{n\rightarrow \infty} \Big\Vert \prod_{j=1}^n \psi_j^n \Big\Vert^\frac 1 n$.

For every $n,m \in \zN_0$ consider the functions $f_n: S_{X^*} \rightarrow \zR\cup \{-\infty\}$ and $f_{n,m}:S_{X^*} \rightarrow \zR$ defined as
$$f_n(\psi) = \log |\langle x_n,\psi\rangle |$$
$$f_{n,m}(\psi)=\max \{f_n(\psi), -m\}.$$
Using that $f_{n,m} \geq f_n$ we obtain
\begin{eqnarray}
\Big\Vert \prod_{j=1}^n  \psi_j^n \Big\Vert^{\frac 1 n} &=& \prod_{j=1}^n  |\langle x_n,\psi_j^n\rangle | ^{\frac 1 n} =
 \exp \left\{ \frac 1 n \sum_{j=1}^n \log \left|\langle x_n,\psi_j^n\rangle \right| \right\}\nonumber \\
&= &\exp \left\{ \frac 1 n \sum_{j=1}^n f_n(\psi_j^n) \right\}=
\exp \left\{  \eta_n(f_n) \right\}\nonumber \\
&\leq &\exp \left\{  \eta_n(f_{n,m}) \right\}.\nonumber \
\end{eqnarray}

Fixed $m$, since $x_n\to x_0$, it is easy to check that the functions $f_{n,m}$  converges uniformly to $f_{0,m}$ as $n\to \infty$. Also, we know that $\eta_n$ $w^*$-converges to $\eta$. This altogether gives that $\eta_n(f_{n,m})$ converges to $\eta(f_{0,m})$ and then
$$\displaystyle\lim_{n\rightarrow \infty} \Big\Vert \prod_{j=1}^n \psi_j^n \Big\Vert^\frac 1 n \leq \exp \left\{  \eta(f_{0,m}) \right\}.$$
This holds for arbitrary m. Since $\mu$ is admissible, taking limit on $m$, we obtain
$$ \displaystyle\lim_{n\rightarrow \infty}\Big\Vert \prod_{j=1}^n \psi_j^n \Big\Vert^\frac 1 n \leq \exp\{\eta(f_{0})\} = \exp\left\{ \int_{S_{X^*}} \log |\langle x_0,\psi\rangle | \,\,\,d\eta(\psi) \right\},$$
as desired.
\end{proof}

\begin{rem} In the previous proof we only use from the Definition \ref{admisible} that
$$\int_{S_{X^*}} \max\{\log|\langle x_0,\varphi\rangle |,-m\} \,\,\,d\eta(\varphi)\rightarrow \int_{X^*} \log|\langle x_0,\varphi\rangle |\,\,\,d\eta(\varphi),$$
that is, we only needed pointwise convergence for the point $x_0 \in S_{(X^*)^*}$, rather than uniform convergence on $S_{(X^*)^*}.$ To see this, it is enough to have
$$\int_{S_{X^*}} \log |\langle x_0,\psi\rangle | \,\,\, d\eta(\psi) < \infty$$
and apply the Dominated Convergence Theorem.
\end{rem}


\subsection*{Upper bounds in Theorem \ref{mainprop}}

For the upper bounds we will obtain a slightly better result, since we will get upper bounds for $\ccc_n(X)$ rather than for $\ccc(X)$. Setting $K$ as the sphere $S_X$ in the following proposition we obtain the upper bounds of Theorem \ref{mainprop}.

\begin{prop}\label{prop polarizacion nesima} Given a finite dimensional Banach space $X$, $K\subseteq B_X$, and an admissible probability measure $\mu$ over $K$, there is a point $\psi_0\in S_{X^*}$, depending on $\mu$, such that
$$\ccc_n(X)\leq \exp\left\{-n \int_{K} \log |\langle x,\psi_0\rangle | \,\,\,d\mu(x) \right\}.$$
\end{prop}

\begin{proof} Consider the function $g:S_{X^*}\rightarrow \zR$ defined as $$g(\psi)=\int_{K} \log |\langle x,\psi\rangle | \,\,\,d\mu(x).$$ We start by showing that $g$ is continuous. For every natural number $m$ define $g_m:S_{X^*}\rightarrow \zR$ by $$g_m(\psi) =\int_{K} \max\{-m,\log |\langle x,\psi\rangle |\} \,\,\,d\mu(x).$$ Given that $\mu$ is admissible, $\{g_m\}_{m\in \zN}$ converges uniformly to $g$ and therefore, since each $g_m$ is con\-ti\-nuous, $g$ is con\-ti\-nuous. Given that $g$ is continuous  and $S_{X^*}$ is compact,  there is $\psi_0\in S_{X^*}$ a global minimum of $g$.

Recall that $\ccc_n(X)$ is the smallest constant such that
$$1=\prod_{j=1}^n \Vert \psi_j \Vert \leq \ccc_n(X) \left\Vert \prod_{j=1}^n  \psi_j \right\Vert$$
for any set of linear functionals $\psi_1,\ldots \psi_n \in S_{X^*}$. So we need to prove that
$$\exp\left\{n  \int_{K} \log |\langle x,\psi_0\rangle | \,\,\,d\mu(x)  \right\} \leq \left\Vert \prod_{j=1}^n  \psi_j \right\Vert.$$

Using that $\mu$ is a probability measure and that $\psi_0$ minimizes  $g$, we obtain
\begin{eqnarray}
\left\Vert \prod_{j=1}^n  \langle \cdot,\psi_j\rangle \right\Vert &\geq & \exp\left\{\log\left( \displaystyle\sup_{x\in K} \prod_{j=1}^n |\langle x,\psi_j\rangle |\right)\right\} \nonumber \\
&=& \exp\left\{ \displaystyle\sup_{x\in K} \sum_{j=1}^n \log |\langle x,\psi_j\rangle |\right\} \nonumber \\
&\geq& \exp\left\{  \int_{K}  \sum_{j=1}^n \log |\langle x,\psi_j\rangle | \,\,\,d\mu(x)\right\} \nonumber \\
&=& \exp\left\{\sum_{j=1}^n  \int_{K} \log |\langle x,\psi_j\rangle | \,\,\,d\mu(x)  \right\}\geq  \exp\left\{n  \int_{K} \log |\langle x,\psi_0\rangle | \,\,\,d\mu(x)  \right\},\nonumber  \
\end{eqnarray}
as desired.
\end{proof}

\begin{rem} In the previous proof we used that $\mu$ is admissible only to prove that $g$ has a global minimum.
\end{rem}



\section{Linear polarization constants of $\ell_p^d$ spaces}\label{sec-elepe}

In this section we apply the method developed in the previous section and stated in Theorem  \ref{mainprop}, to estimate the asymptotic behaviour of the linear polarization constants $\ccc(\ell_p^d(\zK))$. To describe the asymptotic behaviour of two sequences of positive numbers $\{a_d\}_{d\in \zN}$ and $\{b_d\}_{d\in \zN}$ we use the notation $a_d \prec b_d$ to indicate that there is a constant $L>0$ such that $a_d \leq L b_d$. The notation  $a_d \asymp b_d$ means that $a_d \prec b_d$ and $a_d \succ b_d$. In the following we write $dS$ for the normalized surface (Lebesgue) measure over the sphere $S_{\ell_2^d}$.

When we consider a $d$-dimensional (real or complex) Hilbert space $\mathcal{H}_d$, taking in Theorem~\ref{mainprop} both  measures $\mu$ and $\eta$ to be the normalized Lebesgue measure over $S_{\mathcal{H}_d}=S_{{\mathcal{H}^*_d}}$, we recover the following result from \cite{PR}:
\begin{equation}
\ccc(\mathcal{H}_d) = \exp\left\{ -\int_{S_{\mathcal{H}_d}} \log|\langle x,\psi_0\rangle | dS(x) \right\}. \label{hilbert}
\end{equation}
Note that, by symmetry, this expression does not depend $\psi_0$.
If we call $$L(d,\zK )=\int_{S_{\mathcal{H}_d}} \log|\langle x,\psi_0\rangle | dS(x),$$ a standard computation (see \cite{PR}) gives:
\[
	-L(d,\zR)= \left\{ \begin{array}{lcl}
                  \sum_{j=1}^{(d-2)/2} \frac{1}{2j} +\log 2  & \mbox{ if } & d\equiv 0(2) \\																
                    &             &     \\
                  \sum_{j=1}^{(d-3)/2} \frac{1}{2j+1} & \mbox{ if } & d\equiv 1(2)
\end{array}
\right. \mbox{ and } -L(d,\zC)=\frac{1}{2} \sum_{j=1}^{d-1} \frac 1 j.
\]
In particular $\ccc({\mathcal{H}_d}) \asymp \sqrt{d}$. Moreover, using the fact that $ \sum_{j=1}^{d-1} \frac 1 j - 2 \log(\sqrt{d})$ increases monotonically to the Euler-Mascheroni constant $\gamma$ it is easy to see that for $\zK =\zR$ and $d$ even
$$\ccc({\mathcal{H}_d})= e^{-L(d,\zR )} \leq  e^\frac{\gamma}{2}\sqrt{2d},$$
while for the rest of the cases we get											
$$\ccc({\mathcal{H}_d})\leq e^\frac{\gamma}{2}\sqrt{d}.  $$


In order to apply our results to a $d$-dimensional Banach space $X$, we need good candidates for the measures $\eta$ and $\mu$. Ideally, the measure $\eta$ on $S_{X^*}$ should be induced by a sequence of sets of norm one linear functionals $\{\psi_1^n,\ldots,\psi_n^n\}_{n\in \zN}$ such that
$$\Vert \psi_1^n\cdots \psi_n^n \Vert = c_n(X)^{-1}.$$
Since it is not easy to find such functionals, a good guess of their distribution on $S_{X^*}$ would be helpful. When $X$ is a Hilbert space, due to the symmetry of the sphere, it is natural to believe that they are uniformly distributed across the sphere. And that is a good choice:  the measure induced by uniformly distributed functionals is the normalized Lebesgue measure which, as we observed, is an optimal choice of $\eta$.

But this argument is no longer valid for the spaces $\ell_p^d$ with $p\neq 2$. If $\frac{1}{p}+\frac{1}{q}=1$, the lack of symmetry of  $S_{\ell_q^d}$ for $q\ne 2$ suggests that the linear functionals will not be uniformly distributed on the sphere. After some reflection, by the geometry of the sphere, one may expect that if $p<2$, the linear functionals should be  more concentrated around the points $e_1,\ldots,e_n$ than around points of the form $\sum \lambda_i e_i$, with $|\lambda_i|=\frac{1}{d^{1/q}}$. This is the case for $n\leq d$ (see \cite[Theorem 2.4]{CPR}), or for $n=dk$, as we will see below (see proof of Theorem \ref{teo polarizacion}, Step II). For $p>2$ we expect the reverse situation.

Then, for the spaces $\ell_p^d$ we will choose a measure $\eta$ reflecting the previous reasoning and try to obtain the best possible lower bound, taking into consideration that we will not have control over the vector $x_0$ mentioned in Theorem \ref{mainprop}.
%

The following is our main result and gives the asymptotic behaviour of the linear polarization constants $\ccc(\ell_p^d(\zK))$ as $d$ goes to infinity. This extends results of \cite{GV} and \cite{PR} to non-Euclidean spaces.
We devote the rest of this section to its proof.

\begin{thm}\label{teo polarizacion} Let  $1\leq p < \infty$. Then,
\[
\ccc(\ell_p^d(\zK))  \asymp \left\{ \begin{array}{lcl}
                  \sqrt{d}  & \mbox{ if } & p\geq 2 \\																
                    &             &     \\
                  \sqrt[p]{d} & \mbox{ if } & p \leq 2.
\end{array}
\right.
\]
For $p=\infty$ we have the following estimation
$$\sqrt{d} \prec  \ccc(\ell_\infty^d(\zK)) \prec \sqrt{d \log d}.$$
\end{thm}

In order to prove Theorem \ref{teo polarizacion} we need some auxiliary calculations. Next lemma is essentially contained in Lemma 2.8 from \cite{CGP}, but we state it and say a a few words about the proof for completeness.

\begin{lem}\label{concentration} Given $1\leq p <\infty$ we have
$$\int_{S_{\ell_2^d(\zK)}} \| t \|_p^p dS(t)\asymp d^{1-\frac{p}{2}},$$
and for $p=\infty$ we have
$$\int_{S_{\ell_2^d(\zK)}} \| t \|_\infty dS(t)\asymp \left(\frac{\log d}{d}\right)^{\frac 1 2}.$$
\end{lem}
\begin{proof}
The complex case can be easily deduced from the real case, since the norms of $\ell_p^d(\mathbb C)$ and $\ell_p^{2d}(\mathbb R)$ are equivalent up to a factor which is independent of $d$.  The case $p<\infty$ and $\zK=\zR$ is a particular case of Lemma 2.8 from \cite{CGP}.
%
The case $p=\infty$ and $\zK =\zR$ follows just as in the case $p<\infty$, considering the Gaussian measure $\gamma$ over $\zR^d$ and using the well known behaviour of the maximum of $d$ standard Gaussian variables
$$\int_{\zR^d} \Vert z \Vert _\infty \, d\gamma(z) \asymp \sqrt{\log d}. \qedhere$$
\end{proof}

\noindent With this lemma we are able to prove the following.

\begin{lem}\label{lemaux} Let $ 1\leq p < \infty$. Then
$$  \exp\left\{\int_{S_{\ell_2^d(\zK)}} \log \left(\frac{1}{\Vert z \Vert_p}\right) dS(z) \right\}\asymp d^{\frac 1 2 - \frac 1 p} ,$$
and for $p=\infty$ we have
$$  \left(\frac{d}{\log d} \right)^{\frac 1 2} \prec \exp\left\{\int_{S_{\ell_2^d(\zK)}} \log \left(\frac{1}{\Vert z \Vert_\infty}\right) dS(z) \right\}\prec d^{\frac 1 2}.$$
\end{lem}

\begin{proof}

We prove only the real case, since the complex case follows from the real one as in Lemma \ref{concentration}. Let us start with  the upper bound and $p<\infty$, using Jensen's inequality and relation (6.2) from the proof of Theorem 6.1 in \cite{Pi}
$$\int_{S_{\ell_2^d(\zR)}}  \frac{1}{\Vert z \Vert_p^d}dS(z) =\frac{|B_{\ell_p^d}|}{|B_{\ell_2^d}|},$$
we have
\begin{eqnarray}
\int_{S_{\ell_2^d(\zR)}} \log \left(\frac{1}{\Vert z \Vert_p}\right) dS(z)&=&\frac{1}{d} \int_{S_{\ell_2^d(\zR)}} \log \left(\frac{1}{\Vert z \Vert_p^d}\right) dS(z) \nonumber \\
&\leq &\frac{1}{d} \log\left(\int_{S_{\ell_2^d(\zR)}}  \frac{1}{\Vert z \Vert_p^d}dS(z)\right)\nonumber \\
&=& \frac{1}{d}\log \left(\frac{|B_{\ell_p^d}|}{|B_{\ell_2^d}|} \right) = \log \left( \left(\frac{|B_{\ell_p^d}|}{|B_{\ell_2^d}|} \right)^{\frac 1 d} \right). \nonumber \
\end{eqnarray}
Therefore, by \cite[Equation (1.18)]{Pi}
$$|B_{\ell_p^d}|^\frac{1}{d}\asymp \frac{1}{d^p},$$ we obtain
$$\exp\left\{\int_{S_{\ell_2^d(\zR)}} \log \left(\frac{1}{\Vert z \Vert_p}\right) dS(z) \right\} \leq \left(\frac{|B_{\ell_p^d}|}{|B_{\ell_2^d}|} \right)^{\frac 1 d} \prec d^{\frac 1 2 - \frac 1 p}. $$
The upper bound for $p=\infty$  follows using the obvious modifications to the previous reasoning.

For the lower bound and $p<\infty$, we will use again Jensen's inequality to get
\begin{eqnarray}
\int_{S_{\ell_2^d(\zR)}} \log \left(\frac{1}{\Vert z \Vert_p}\right) dS(z) &=& \int_{S_{\ell_2^d(\zR)}} -\frac{1}{p}\log \left(\Vert z \Vert_p^p\right)  dS(z)  \nonumber \\
&\geq & -\frac{1}{p}\log \left(\int_{S_{\ell_2^d(\zR)}} \Vert z \Vert_p^p  dS(z)\right).   \nonumber \
\end{eqnarray}
Then, using Lemma \ref{concentration}, we obtain
\begin{eqnarray}
\exp\left\{\int_{S_{\ell_2^d(\zR)}} \log \left(\frac{1}{\Vert z \Vert_p}\right) dS(z) \right\} &\geq & \left(\int_{S_{\ell_2^d(\zR)}} \Vert z \Vert_p^p  dS(z)\right)^{-\frac{1}{p}} \nonumber \\
&\succ & d^{\frac 1 2 - \frac 1 p}. \nonumber \
\end{eqnarray}

As before, using the obvious modifications to the previous reasoning, we obtain the lower bound for the case  $p=\infty$.
 \end{proof}

%


Now we are ready to prove our main result.

\begin{proof}[Proof of Theorem \ref{teo polarizacion}] In order to have a better organization, we divide the proof in different parts.  Given that the proof is the same for $\zK=\zC$ or $\zR$, for simplicity, we omit the notation on the scalar field. Throughout this proof $q$ will be the conjugate exponent of $p$.
\paragraph{\textbf{Step I:}  $\ccc(\ell_p^d) \succ \sqrt{d}$ for $2 < p  \leq \infty$.}

As mentioned before, we want to consider a measure related to the geometry of the sphere $S_{\ell_p^d}$. That being said, we also want a measure that can be easily related to the Lebesgue measure of $S_{\ell_2^d}$, given that for Hilbert spaces the linear polarization constant is known.

Consider, then, the measure $\eta$ on $S_{\ell_q^d}$ defined by
$$\eta(A) = \int_{H(A)} \frac{1}{|DH^{-1}(\varphi)|} dS(\varphi),$$
where  $H:S_{\ell_q^d}\rightarrow S_{\ell_2^d}$ is defined as $H(\psi)=\frac{\psi}{\Vert \psi \Vert_2}$. That is, we choose $\eta$ such that for any integrable function $f:S_{\ell_q^d}\rightarrow \zK$, we have
\begin{equation}\label{ecuacion eta}
\int_{S_{\ell_q^d}} f(\psi) \,\,\,d\eta(\psi) = \int_{S_{\ell_2^d}} f\left(\frac{\varphi}{\Vert \varphi\Vert_q}\right) dS(\varphi).
\end{equation}

Using that the normalized Lebesgue measure is admissible, and its close relation with $\eta$, it is easy to see that $\eta$ is admissible. Then, by Theorem \ref{mainprop}, there is $x_0\in S_{\ell_p^d}$ such that

$$\ccc(\ell_p^d) \geq \exp\left\{ -\int_{S_{\ell_q^d}} \log (|\langle x_0,\psi\rangle |) \,\,\,d\eta(\psi) \right\}.$$
Let's find an upper bound for the integral. By \eqref{ecuacion eta}, we have
\begin{eqnarray}
 \int_{S_{\ell_q^d}} \log (|\langle x_0,\psi\rangle |) \,\,\,d\eta(\psi) &= & \int_{S_{\ell_2^d}} \log \left(\left|\langle x_0,\frac{\varphi}{\Vert \varphi \Vert_q}\rangle \right|\right) dS(\varphi) \nonumber \\
&= & \int_{S_{\ell_2^d}} \log \left(\left|\langle x_0\frac{\Vert x_0 \Vert_2}{\Vert x_0\Vert_2},\frac{\varphi}{\Vert \varphi \Vert_q}\rangle \right|\right) dS(\varphi)  \nonumber \\
&= &\int_{S_{\ell_2^d}} \log \left(\left|\langle \frac{x_0}{\Vert x_0 \Vert_2},\varphi\rangle \right|\right) dS(\varphi) \nonumber \\
&  & +\int_{S_{\ell_2^d}} \log \left(\frac{1}{\Vert \varphi \Vert_q}\right) dS(\varphi)+ \log  (\Vert x_0 \Vert_2). \nonumber \
\end{eqnarray}
Then, using \eqref{hilbert}, Lemma \ref{lemaux} and that $x_0 \in S_{\ell_p^d}$, with $p>2$, we obtain
\begin{eqnarray}
\ccc(\ell_p^d) &\geq & \ccc(\ell_2^d) \exp\left\{ -\int_{S_{\ell_2^d}} \log \left(\frac{1}{\Vert \varphi \Vert_q}\right) dS(\varphi) \right\} \frac{1}{\Vert x_0 \Vert_2}\nonumber \\
&\succ &\ccc(\ell_2^d) d^{\frac 1 q - \frac 1 2} d^{\frac 1 p - \frac 1 2} =  \ccc(\ell_2^d) \nonumber\\
&\asymp & \sqrt{d}.\nonumber\
\end{eqnarray}

\paragraph{\textbf{Step II:} $\ccc(\ell_p^d) \succ \sqrt[p]{d}$ for $p < 2$.}

Note that in this case, the previous procedure would lead us to
$\ccc(\ell_p^d) \succ \sqrt[q]{d},$
so we need an alternative way. It is enough to find a subsequence of natural numbers $\{n_k\}_{k\in \zN}$ such that
$$\ccc_{n_k} (\ell_p^d) \succ d^{\frac{n_k}{ p}}.$$
Let us consider the subsequence $n_k=dk$. For each $k$ consider the following set of norm one linear functionals
$$\{\underbrace{e_1,\ldots, e_1,}_{k \mbox{ times}}\ldots ,\underbrace{e_d,\ldots ,e_d}_{k \mbox{ times}}\}\subseteq S_{\ell_q^d},$$
that is, we consider $k$ copies of each vector of the canonical basis. Then we have
\begin{equation}
\ccc_{n_k} (\ell_p^d) \geq \Vert   (e_1)^k\cdots (e_d)^k \Vert^{-1} = \sqrt[p]{ \frac{ \left( k+\cdots +k \right)^{k+\cdots +k}}{ k^k\cdots k^k } }= \sqrt[p] { d^{n_k} }, \nonumber
\end{equation}
since $(e_1)^k\cdots (e_d)^k $ attains its maximum on $\left(\frac{1}{d^{\frac{1}{p}}},\ldots, \frac{1}{d^{\frac{1}{p}}} \right)$.

Note that in this case we proved that  $\ccc(\ell_p^d) \geq \sqrt[p]{d}$, rather than $\ccc(\ell_p^d) \succ \sqrt[p]{d}.$ We also remark that the strategy followed in this step would not give useful information in the previous case.

\paragraph{\textbf{Step III:} $\ccc(\ell_p^d) \prec \sqrt{d}$ for $2<p <\infty$.}

As before, define the measure $\mu$ on $S_{\ell_p^d}$ by
$$\mu(A)=\int_{G(A)} \frac{1}{|DG^{-1}(z)|} dS(z),$$
where  $G:S_{\ell_p^d}\rightarrow S_{\ell_2^d}$ is defined as $G(z)=\frac{z}{\Vert z \Vert_2}$.
Proceeding as in the previous case, we obtain
\begin{equation}\label{ecuacion cota sup del teo}
\ccc(\ell_p^d) \leq \ccc(\ell_2^d) \exp\left\{ -\int_{S_{\ell_2^d}} \log \left(\frac{1}{\Vert z \Vert_p}\right) dS(z) \right\} \frac{1}{\Vert \psi_0 \Vert_2},
\end{equation}
where $\psi_0$ is some point in $S_{\ell_q^d}$. Note that so far the fact that $2<p <\infty$ has not been used.

Using Lemma \ref{lemaux} and the that $q<2$ we conclude
\begin{eqnarray}
\ccc(\ell_p^d) &\prec & \ccc(\ell_2^d) d^{\frac 1 p - \frac 1 2} d^{\frac 1 q - \frac 1 2}  =  \ccc(\ell_2^d) \nonumber \\
& \asymp & \sqrt{d} .\nonumber
\end{eqnarray}

\paragraph{\textbf{Step IV:} $\ccc(\ell_\infty^d) \prec \sqrt{d \log d}$.}

Combining \eqref{ecuacion cota sup del teo} with Lemma \ref{lemaux} for $p=\infty$  we obtain
$$
\ccc(\ell_\infty^d) \prec  \ccc(\ell_2^d)  \left(\frac{\log d}{d} \right)^{\frac 1 2} d^{\frac 1 2}\nonumber
 =  \sqrt{d \log d} .
$$

\paragraph{\textbf{Step V:} $\ccc(\ell_p^d) \prec \sqrt[p]{d}$ for $p < 2$.}
By \eqref{ecuacion cota sup del teo}, the fact that in this case  $\psi_0$ is some point in $S_{\ell_q^d}$ with $q>2$ and  Lemma~\ref{lemaux} we obtain
$$
\ccc(\ell_p^d) \prec  \ccc(\ell_2^d) d^{\frac 1 p - \frac 1 2} 1   \asymp  \sqrt[p]{d}. \nonumber \
$$

\end{proof}



\section{On the $n$th linear polarization constant of $\ell_\infty^d(\zC)$}\label{seccion infinito}

In this section we study the $n$th linear polarization constant of the complex finite dimensional spaces $\ell_\infty^d(\zC)$. Although we do not solve the gap in Theorem \ref{teo polarizacion}, we obtain a more precise result on the lower bounds.  We use  a probabilistic approach to prove the existence of li\-near functionals $\varphi_1,\ldots,\varphi_n :\ell_\infty^d\rightarrow \zC$ such that the norm of the product is small in comparison with the product of the norms.
The probabilistic techniques we use in this section are an adaptation to our problem of techniques used, for example, by H. Boas in \cite{Bo}. The aim of this section is then to prove the following.

\begin{prop}\label{prop infinito} The $n$th linear polarization constant of  $\ell_\infty^d(\zC)$ satisfies
$$\ccc_n(\ell_\infty^d(\zC)) \geq \frac{1}{2}\sqrt{\frac{ d^n}{(24n)^d}}.$$
\end{prop}

\begin{rem} Note that in particular, the result from above assures us that
$$
\ccc(\ell_\infty^d(\zC)) \geq  \sqrt{d}. \nonumber
$$
This improves the bound from Theorem \ref{teo polarizacion} where we had $\ccc(\ell_\infty^d(\zC))  \succ  \sqrt{d}.$
\end{rem}

We start by using some notation. 
Let $\{\ve_k^j:\Omega \rightarrow \zR\}_{j,k}$, with $j\in\{1,\ldots,n\}$ and $k\in \{1,\ldots, d\}$, be a family of independent Rademacher functions  over a probability space $(\Omega,\Sigma,P)$.  That is, $\{\ve_k^j\}_{j,k}$ are independent random variables such that $P(\ve_k^j=1)=P(\ve_k^j=-1)=\frac 1 2$ for $j=1,\ldots,n$ and $k=1,\ldots, d$. For any $t\in \Omega$ and $j\in\{1,\ldots,n\}$ we define the linear function $\varphi_j(\cdot,t):\ell_\infty^d\rightarrow \zC$ as $\varphi_j(z,t) =\sum_{k=1}^d \ve_k^j(t) z_k$ and $F:\ell_\infty^d\times \Omega \rightarrow \zC$ by
$$F(z,t) =\prod_{j=1}^n\varphi_j(z,t)=\sum_{k_1,\ldots,k_n=1}^d \ve_{k_1}^1\cdots \ve_{k_n}^nz_{k_1}\cdots z_{k_n}.$$

We will show the existence of some $t_0 \in \Omega$ such that the norm $\left\Vert \prod_{j=1}^n\varphi_j(\cdot,t_0) \right\Vert =\Vert  F(\cdot,t_0)\Vert$ is small.
To do this we need some auxiliary lemmas related to the function $F$, the geometry of the $d$ dimensional torus $\mathbb{T}^d=\{z\in \ell_\infty^d(\zC) : |z_k| = 1\}$ and the space $\ell_\infty^d(\zC)$.

\begin{lem}\label{aux primero}
For any natural number $N$, the $d$-dimensional torus $\mathbb{T}^d$ can be covered up with $N^d$ balls of $\ell_\infty^d(\zC)$, with center on $\mathbb{T}^d$ and radius $\frac{\pi}{N}$.
\end{lem}
\begin{proof} It is enough to consider the balls of center $(e^{2\pi i\frac{j_1}{N}},\ldots, e^{2\pi i\frac{j_d}{N}})$, with $j_1,\ldots,j_d \in \{1,\ldots,N\}$.
\end{proof}

\begin{lem}\label{aux segundo} Given  $z \in \mathbb{T}^d$ and  a positive number $R$,  we have
$$P(|F(z,t)|>R)\leq \frac{1}{R^2}d^n.$$
\end{lem}

\begin{proof}
%
%
%
%
If we write $z=(z_1,\ldots,z_d)$, then the expected value of $\vert F(z, \cdot) \vert^2$ is
\begin{eqnarray}
\zE(|F(z,\cdot)|^2)&=&\zE\left(\left|\sum_{k_1,\ldots,k_n=1}^d \ve_{k_1}^1\cdots \ve_{k_n}^n z_{k_1}\cdots z_{k_n}\right|^2\right) \nonumber \\
&=&\sum_{k_1,\ldots,k_n=1}^d | z_{k_1}\cdots z_{k_n}|^2 =d^n, \nonumber \
\end{eqnarray}
where we used the independence of the family $\{\ve_k^j\}_{j,k}$. The result now follows from Chebyshev's inequality.
\end{proof}
\begin{lem}\label{aux tercero} For any pair of norm one vectors $z,w\in \ell_\infty^d(\zC)$ and any $t\in \Omega$, we have
$$|F(w,t)-F(z,t)| \leq n \, e\, \|F(\cdot,t)\| \|w-z\|.$$
\end{lem}
\begin{proof} If we define $\gamma(s)=ws +z (1-s)$ for $0\leq s\leq 1$,  there is $0\leq c\leq 1$ such that
\begin{eqnarray}
|F(w,t)-F(z,t)|&= & |F(\gamma(1),t)-F(\gamma(0),t)| \nonumber \\
& = &|D F(\gamma(c),t)\circ D\gamma (c)| \nonumber \\
& = &\frac{n^n}{(n-1)^{n-1}}\Vert  F(\cdot,t)\Vert  \Vert \gamma(c) \Vert^{n-1} \Vert D\gamma(c) \Vert \label{h inequality} \\
& \leq & n e \Vert  F(\cdot,t)\Vert   \Vert w-z \Vert \nonumber \
\end{eqnarray}
where in \eqref{h inequality} we have used the following inequality, which is a particular case of a result by Harris~\cite[Corollary~1]{H}: if $P:X\rightarrow \zC$ is an $n$-homogeneous polynomial over a complex Banach space $X$, then
$$\Vert DP \Vert \leq \frac{n^n}{(n-1)^{n-1}} \Vert P \Vert.$$
\end{proof}

\begin{lem}\label{aux cuarto} For any positive number $R$,
$$P(\| F(\cdot,t) \| > 2R) < (24n)^d\frac{d^n}{R^2} .$$
\end{lem}

\begin{proof} By Lemma \ref{aux primero}, there is a family of points $\{w_1,\ldots,w_{(24n)^d}\}\subseteq \mathbb{T}^d$ such that for any $z\in \mathbb{T}^d$,  we have
$$\Vert w_i-z\Vert\leq\frac{\pi}{24n} < \frac{1}{2ne} $$
for some $i=1,\ldots, (24n)^d$.
For any fixed $t\in \Omega$, by the maximum modulus principle, there is $z_0\in \mathbb{T}^d$ such that
$$\|F(\cdot,t)\| =|F(z_0,t)|.$$
Let $i$ be such that $\| w_i - z_0  \| \leq \frac{1}{2ne}$. By Lemma \ref{aux tercero}
$$|F(w_i,t)-F(z_0,t)| \leq \|F(\cdot,t)\| n e \|w_i-z_0\| < \|F(\cdot,t)\| \frac{1}{2}.$$
Therefore, for each $t$ we have
$$\frac{\|F(\cdot,t)\| }{2} < |F(w_i,t)|$$
for some $i$, and then we conclude that
$$\|F(\cdot,t)\| < 2 \max_i \{|F(w_i,t)|:i=1,\ldots,(24n)^d\}.$$
Since $t\in \Omega$ was arbitrary, using Lemma \ref{aux segundo}, we have
\begin{eqnarray}
P(\| F(\cdot,t) \| > 2R) &<& P(\max_i \{ |F(w_i,t)|:i=1,\ldots,(24n)^d\} > R) \nonumber \\
&\leq& \sum_{i=1}^{(24n)^d} P(|F(w_i,t)| > R) \nonumber \\
&\leq& (24n)^d\frac{d^n}{R^2} , \nonumber
\end{eqnarray}
as desired.
\end{proof}

\noindent Now we are ready to prove the main result of this section.

\begin{proof}[Proof of Proposition~\ref{prop infinito}]
Take in Lemma \ref{aux cuarto}
$$R= \sqrt{(24n)^d d^n}.$$
Then
$$
P (\| F(\cdot,t) \| > 2R) < 1.
$$
Therefore, there is $t_0 \in \Omega$, such that
\begin{eqnarray}
\| \prod_{j=1}^n \varphi_j(\cdot, t_0) \| &=& \| F(\cdot,t_0) \| \leq 2R \nonumber \\
&=& 2\sqrt{(24n)^d d^n} = 2\sqrt{\frac{(24n)^d}{ d^n}}d^n \nonumber \\
&=& 2\sqrt{\frac{(24n)^d}{ d^n}}\prod_{j=1}^n \|\varphi_j(\cdot, t_0) \|, \
\end{eqnarray}
which ends the proof.
\end{proof}



\end{document}